\documentclass[12pt]{article}
\usepackage{amscd, amssymb, latexsym, amsmath}

\newtheorem{thm}{Theorem}
\newtheorem{cor}{Corollary}
\newtheorem{pro}{Proposition}
\newtheorem{lem}{Lemma}

\newenvironment{proof}
{\noindent {\em Proof.}} {\hfill $\Box$}

\numberwithin{thm}{section} \numberwithin{cor}{section}
\numberwithin{pro}{section}
\numberwithin{lem}{section} \numberwithin{equation}{section}

\newcommand{\R}{\mathbb R}

\numberwithin{equation}{section}

\begin{document}
\title{Limit of quasilocal mass at spatial infinity}
\author{Mu-Tao Wang and Shing-Tung Yau}

\date{May 31, 2009, this version Jun 8, 2009}
\maketitle
\begin{abstract}
We study the limit of quasilocal mass defined in \cite{wy1} and \cite{wy2} for a family of spacelike 2-surfaces in spacetime. In particular, we show the limit coincides with the ADM mass at spatial infinity. The limit for coordinate spheres of a boosted slice of the Schwarzchild solution is computed explicitly and shown to give the expected energy-momentum four-vector.
\end{abstract}

\section{Review of the definition of quasilocal energy}

\footnote{We would like to thank PoNing Chen for his help in checking the correctness of the calculations in \S 3.}
\footnote{The first author is supported by NSF grant DMS-0605115 and the second author is supported by NSF
grant DMS-0628341.} In \cite{wy1} and \cite{wy2}, we define a notion of quasilocal mass for spacelike 2-surfaces in a spacetime. Given an isometric embedding of a 2-surface into $\R^{3,1}$ and a future timelike unit vector (observer)
in $\R^{3,1}$, we associated a quasilocal energy with respect to a canonical gauge. Minimizing among the reference data gives the quasilocal mass and the quasilocal energy-momentum four-vector. We prove that the mass has the important positivity property and it vanishes for surfaces in $\R^{3,1}$.
The expression for the mass is nevertheless rather nonlinear and complicated. In this article, we show that for a family of surfaces going out to spatial infinity, the expression indeed gets ``linearized" and gives a well-defined energy-momentum four-vector.

First of all, we recall the definition of quasilocal energy in \cite{wy1}. Let $\Sigma$ be a spacelike 2-surface in a time-orientable spacetime $N$. Consider a reference isometric embedding $\Sigma\hookrightarrow \R^{3,1}$. Fix a future timelike unit vector $t_0^\nu$ in $\R^{3,1}$. We decompose $t_0^\nu$ along $\Sigma\subset\R^{3,1}$ into $t_0^\nu=N_0 u_0^\nu+N_0^\nu$ in which $N_0$ is the lapse function, $N_0^\nu$ is the shift vector, and $u_0^\nu$ is the future timelike unit normal vector field along $\Sigma \subset \R^{3,1}$ determined by this decomposition. We also take the spacelike outward pointing unit normal $v_0^\nu$ that is orthogonal to $u_0^\nu$ along $\Sigma\subset\R^{3,1}$. $(u_0^\nu, v_0^\nu)$ is the reference gauge for $\Sigma\subset \R^{3,1}$ with respect to $t_0^\nu$. To compute the quasilocal energy, we also need the canonical gauge $(\bar{u}^\nu, \bar{v}^\nu)$ along $\Sigma \subset N$. $\bar{u}^\nu$ is characterized as the  unique future timelike unit normal vector field along $\Sigma\subset N$ such that
\begin{equation}\label{canon}h_\nu\bar{u}^\nu=(h_0)_\nu u_0^\nu,\end{equation} where $h^\nu$ is the mean curvature vector of $\Sigma \subset N$ and $h_0^\nu$ is the mean curvature vector of $\Sigma\subset \R^{3,1}$. $\bar{v}^\nu$ is the spacelike unit normal vector that is orthogonal to $\bar{u}^\nu$ and satisfies $\bar{v}^\nu h_\nu<0$.
Take a spacelike hypersurface $\Omega_0 \subset \R^{3,1}$ spanned by $\Sigma\subset \R^{3,1}$ and $v_0^\nu$, and a spacelike hypersurface $\bar{\Omega}\subset N$ spanned by $\Sigma \subset N$ and $\bar{v}^\nu$. Let $k_0$ be the mean curvature of $\Sigma$ with respect to $\Omega_0$ and $\bar{k}$ be the mean curvature of $\Sigma$ with respect to $\bar{\Omega}$. Also denote by $(K_0)_{\mu\nu}$ and $\bar{K}_{\mu\nu}$ the extrinsic curvatures of $\Omega_0$ and $\bar{\Omega}$, respectively.  These data depend only on the gauges along $\Sigma$ but not on the hypersurfaces. Quasilocal energy in the canonical gauge (see equation (6) in \cite{wy1}) is defined to be \begin{equation}\frac{1}{8\pi}\int_\Sigma (k_0-\bar{k})N_0-(v_0^\mu(K_0)_{\mu\nu}-\bar{v}^\mu\bar{K}_{\mu\nu})N_0^\nu.\end{equation}
 
  We shall rewrite the quasilocal energy in terms of the mean curvature gauge. In order to do so, we adopt a different set of notations from \cite{wy2}. Set $T_0=t_0^\nu$, $H_0=h_0^\nu$, $H=h^\nu$, $\breve{e}_3=v_0^\nu$, $\breve{e}_4=u_0^\nu$, $\bar{e}_3=\bar{v}^\nu$, $\bar{e}_4=\bar{u}^\nu$. Denote by $X:\Sigma\rightarrow \R^{3,1}$ the position vector of the isometric embedding  and by $\tau=\langle X, T_0\rangle $ the restriction of the time function associated with $T_0$. $T_0=\sqrt{1+|\nabla\tau|^2}\breve{e}_4-\nabla\tau$ and thus $N_0=\sqrt{1+|\nabla \tau|^2}$ and $N_0^\nu=-\nabla \tau$.
The quasilocal energy becomes
\begin{equation}\label{qle1}\frac{1}{8\pi}\int_\Sigma (-\langle H_0, \breve{e}_3\rangle+\langle H, \bar{e}_3\rangle)\sqrt{1+|\nabla\tau|^2}-(\langle \nabla^{\R^{3,1}}_{-\nabla\tau} \breve{e}_4, \breve{e}_3\rangle-\langle \nabla^{N}_{-\nabla\tau} \bar{e}_4, \bar{e}_3\rangle).\end{equation}

Suppose the mean curvature vector $H_0$ of $\Sigma$ in $\R^{3,1}$ is spacelike. Let $e_3^{H_0}=\frac{-H_0}{|H_0|}$ be the unit vector in the direction of $-H_0$ and $e_4^{H_0}$ the future-directed time-like unit normal vector with $\langle e_3^{H_0}, e_4^{H_0}\rangle=0$. The relation between the two gauges is
\[e_3^{H_0}=\cosh\theta_0 \breve{e}_3+\sinh\theta_0 \breve{e}_4,\,\,\text{and}\,\,
e_4^{H_0}=\sinh\theta_0 \breve{e}_3+\cosh\theta_0 \breve{e}_4\] for some $\theta_0\in \R$. Since $\Delta\tau=-\langle H_0, T_0\rangle $, we derive  \begin{equation}\label{sinh_theta}\sinh\theta_0=\frac{-\Delta\tau}{|H_0|\sqrt{1+|\nabla\tau|^2}}.\end{equation} Therefore, 
\[\langle {\nabla}^{\R^{3,1}}_{\nabla\tau} \breve{e}_4,  \breve{e}_3\rangle=-\nabla\theta_0 \cdot\nabla\tau+
\langle {\nabla}^{\R^{3,1}}_{\nabla\tau} e_4^{H_0},  e_3^{H_0}\rangle.\]

The canonical gauge condition (\ref{canon})
\[\langle H_0, \breve{e}_4\rangle=\langle H, \bar{e}_4\rangle\] implies $e^H=\frac{-H}{|H|}$ is given by
\[e_3^{H}=\cosh\theta \bar{e}_3+\sinh\theta \bar{e}_4 \text{ with } \sinh\theta=\frac{-\Delta\tau}{|H|\sqrt{1+|\nabla\tau|^2}}.\] Expression (\ref{qle1}) can now be rewritten in terms of the mean curvature gauge.

To summarize, let $\Sigma \subset N$ be a spacelike 2-surface in a spacetime $N$ and let $X:\Sigma\hookrightarrow \R^{3,1}$ be a reference isometric embedding of $\Sigma$ into the Minkowski space. For any given future timelike constant unit vector $T_0\in \R^{3,1}$, the time function on $\Sigma$ is denoted by $\tau=-\langle X, T_0\rangle$. Let ${H}$ be the mean curvature vector of $\Sigma$ in $N$, we assume $H$ is spacelike.  Let ${J}$ be the future timelike unit normal vector field along $\Sigma$ in $N$ which is dual to ${H}$ along the light cone in the normal bundle of $\Sigma$ in $N$. Denote by ${H}_0$ and ${J}_0$ the corresponding data on the isometric embedding in $\R^{3,1}$. Again, $H_0$ is assume to be spacelike in $\R^{3,1}$. The quasilocal energy of $\Sigma$ with respect to the pair $(X, T_0)$ is given by
\begin{equation}\label{qle}\begin{split}E(\Sigma, X, T_0)&=\frac{1}{8\pi} \int_\Sigma \sqrt{|{H}_0|^2(1+|\nabla\tau|^2)+(\Delta\tau)^2}-\sqrt{|{H}|^2(1+|\nabla\tau|^2)+(\Delta\tau)^2}\\
&-\Delta \tau \left[\sinh^{-1}(\frac{\Delta\tau}{\sqrt{1+|\nabla\tau|^2}|{H}_0|})-\sinh^{-1}(\frac{\Delta\tau}{\sqrt{1+|\nabla\tau|^2}|{H}|})\right]\\
&-\langle\nabla^{\R^{3,1}}_{\nabla\tau} \frac{{J}_0}{|{H}_0|}, \frac{{H}_0}{|{H}_0|}\rangle +\langle\nabla^{N}_{\nabla\tau} \frac{{J}}{|{H}|}, \frac{{H}}{|{H}|}\rangle  dv_\Sigma,\end{split}\end{equation} where $\Delta\tau$ is the Laplacian of $\tau$ on $\Sigma$ (with respect to the induced metric), and $\nabla^N$ and $\nabla^{\R^{3,1}}$ are the covariant derivatives on $N$ and $\R^{3,1}$, respectively, and $\nabla \tau$ is the gradient of $\tau$ on $\Sigma$ (with respect to the induced metric again), considered as a tangent vector field on $\Sigma$. In the expressions for the last two integrands, we push forward $\nabla \tau$ by the embeddings and identify it as vector fields along $\Sigma$ in $\R^{3,1}$ and $N$, respectively.

\section{General formula for the limit of quasilocal energy}
Fix $R>0$ and suppose $\Sigma_{r}$, $R<r<\infty$, is a family of closed 2-surfaces in $N$, and $X_{r}$ is a family of isometric embeddings of $\Sigma_{r}$ into $\R^{3,1}$.  In the following theorem, we derive an expression for the limit of $E(\Sigma_{r}, X_{r}, T_0)$.

\begin{thm}
Suppose the mean curvature vectors of $\Sigma_{r}$  and of the image of $X_{r}$ in $\R^{3,1}$ are both spacelike for $r>R_0$ and $\frac{|H|}{|{H}_{0}|}\rightarrow 1$ as $r\rightarrow \infty$. Then the limit of $E(\Sigma_{r}, X_{r}, T_0)$ as $r\rightarrow \infty$ (if exists)  is the same as the limit of

\[\frac{1}{8\pi}  \int_{\Sigma_{r}} -\langle T_0,  \frac{{J}_0}{|H_0|}\rangle (|H_0|-|H|) - \langle\nabla^{\R^{3,1}}_{\nabla\tau} \frac{{J}_0}{|H_0|}, \frac{H_0}{|H_0|}\rangle+ \langle\nabla^{N}_{\nabla\tau} \frac{{J}}{|H|}, \frac{H}{|H|}\rangle  dv_{{r}}.\]
\end{thm}

\begin{proof}
We compute
\begin{equation}\label{Delta_tau}\Delta \tau=-H_0\cdot T_0=|H_0|\langle e^{H_0}_3, T_0\rangle \end{equation}
and
\begin{equation}\label{nabla_tau}|\nabla \tau|^2=-1+ \langle e^{H_0}_4, T_0\rangle^2-\langle e^{H_0}_3, T_0\rangle^2\end{equation} where $e^{H_0}_3=\frac{-H_0}{|H_0|}$ and $e^{H_0}_4= \frac{J_0}{|H_0|}$ is the future timelike unit normal dual to $e^{H_0}_3$ along the image of $X$ in $\R^{3,1}$.

Rationalize the expression
\[\sqrt{|H_0|^2(1+|\nabla\tau|^2)+(\Delta\tau)^2}-\sqrt{|H|^2(1+|\nabla\tau|^2)+(\Delta\tau)^2}\] as
\[(|H_0|-|H|)\frac{(|H_0|+|H|)(1+|\nabla\tau|^2)}{\sqrt{|H_0|^2(1+|\nabla\tau|^2)+(\Delta\tau)^2}
+\sqrt{|H|^2(1+|\nabla\tau|^2)+(\Delta\tau)^2}}.\]

By assumption $\frac{|H|}{|H_0|}\rightarrow 1$ at infinity, the limit as $r\rightarrow \infty$ is thus the same as the limit of
\begin{equation}\label{first}\frac{1}{8\pi}\int_{\Sigma_{r}} \frac{\langle e^{H_0}_4, T_0\rangle^2-\langle e^{H_0}_3, T_0\rangle^2}{-\langle e^{H_0}_4,  T_0\rangle} (|H_0|-|H|) dv_{\Sigma_{r}}.\end{equation}
Next we study the term
\[-\Delta \tau \left[\sinh^{-1}(\frac{\Delta\tau}{\sqrt{1+|\nabla \tau|^2}|H_0|})-\sinh^{-1}(\frac{\Delta\tau}{\sqrt{1+|\nabla\tau|^2}|H|})\right]\]
by rewriting it as
\[-\Delta \tau \left[\sinh^{-1}(\frac{\Delta\tau}{\sqrt{1+|\nabla \rho|^2}|H_0|})-\sinh^{-1}(\frac{\Delta\tau}{\sqrt{1+|\nabla\tau|^2}|H_0|}\frac{|H_0|}{|H|})\right].\]
Note that
\[\frac{\sinh^{-1} A-\sinh^{-1} (A(1+x))}{x}\rightarrow \frac{-A}{\sqrt{1+A^2}}\] as $x\rightarrow 0$. With $x=\frac{|H_0|}{|H|}-1\rightarrow 0$, the limit of the second term is thus the same as the limit of
\begin{equation}\label{second}\frac{1}{8\pi}\int_{\Sigma_{r}} \frac{\langle e^{H_0}_3, T_0\rangle^2}{-\langle e^{H_0}_4, T_0\rangle} (|H_0|-|H|) dv_{{r}}.\end{equation}
The theorem is proved by combining (\ref{first}) and (\ref{second}).
\end{proof}

Suppose the image of the isometric embedding lies $X_r$ in $\R^3\subset \R^{3,1}$, then $e^{H_0}_4=\frac{{J}_0}{|H_0|}$ is a constant vector and the $\langle\nabla^{\R^{3,1}}_{\nabla\tau} \frac{{J}_0}{|H_0|}, \frac{H_0}{|H_0|}\rangle$ term vanishes. In this case, $e^{H_0}_3$ coincide with the outward unit normal of the embedding in $\R^3$.

\begin{cor}
Suppose the reference isometric embedding is in $\R^3\subset\R^{3,1}$ and $\frac{|H|}{|H_0|}\rightarrow 1$ as $r\rightarrow \infty$, then the limit of the
quasilocal energy with respect to $T_0=(\sqrt{1+|a|^2}, a^1, a^2, a^3)$ with $|a|^2=\sum_{i=1}^3 (a^i)^2$ is
\begin{equation} \label{limit_qle}(\sqrt{1+|a|^2}) \frac{1}{8\pi} \int_{\Sigma_{r}} |H_0|-|H| dv_{{r}}+\frac{1}{8\pi}\int_{\Sigma_{r}} \langle\nabla^{N}_{\nabla\tau} \frac{{J}}{|H|}, \frac{H}{|H|}\rangle  dv_{{r}}.\end{equation}
\end{cor}

Suppose the isometric embedding for $\Sigma_r$ is given by $X_r =(X^1, X^2, X^3):\Sigma\rightarrow \R^3$ and consider $X^i, i=1, 2, 3$ as functions on $\Sigma_r$. Thus
$\nabla \tau=-\sum_{i=1}^3 a^i \nabla X^i$ and we obtain a limiting quasilocal energy-momentum four-vector $(e, p_1, p_2, p_3)$ as the limit of
\begin{equation}\label{qlem}\begin{split} e&=\lim_{r\rightarrow \infty}\frac{1}{8\pi} \int_{\Sigma_{r}} |H_0|-|H| dv_{{r}}\\
 p_i &=\lim_{r\rightarrow \infty}\frac{1}{8\pi}\int_{\Sigma_{r}} \langle\nabla^{N}_{-\nabla X^i} \frac{{J}}{|H|}, \frac{H}{|H|}\rangle  dv_{{r}}, i=1, 2, 3.\end{split} \end{equation}
\section{Relating to ADM energy-momentum}
Let $(M, g_{ij}, p_{ij})$ be an asymptotically flat hypersurface in a spacetime $N$. Thus there exists a compact set $K\subset M$ such that $M\backslash K$ is diffeomorphic to a union of complements of balls in $\R^3$ (ends) such that
 $g_{ij}=\delta_{ij} +a_{ij}$  with
$a_{ij}=O(\frac{1}{r})$, $\partial_k(a_{ij})=O(\frac{1}{r^2})$,
$\partial_l\partial_k(a_{ij})=O(\frac{1}{r^3}),$ and  $p_{ij}=O(\frac{1}{r^2})$,
$\partial_k(p_{ij})=O(\frac{1}{r^3})$ on each end of $M\backslash K$.

The ADM energy momentum (Arnowitt-Deser-Misner) of an end of $M$  is the four vector  $(E, P_1, P_2,P_3)$ where
\[E=\lim_{r\rightarrow \infty}\frac{1}{16\pi}
\int_{S_r}(\partial_j g_{ij}-\partial_i g_{jj})\nu^i dv_r\] is the total energy and
\[P_k=\lim_{r\rightarrow \infty} \frac{1}{16\pi}\int_{S_r}
2(p_{ik}-\delta_{ik} p_{jj})\nu^i dv_r\] is the total momentum. Here $S_r$ is a coordinate sphere of radius $r$ on the end and $\nu$ is the outward unit normal of $S_r$.

The positive mass theorem (Schoen-Yau \cite{sy1}, Witten \cite{wi}) asserts that under the dominant energy condition, the four-vector $(E, P_1, P_2,P_3)$ is
future timelike, i.e.  \[E\geq 0 \,\,\,\text{and}\,\,
-E^2+P_1^2+P_2^2+P_3^2\leq 0.\]

In the following, we prove that for coordinate spheres of radius $r$, the limit of the quasilocal energy momentum (\ref{qlem}) is the same as the ADM energy-momentum.

\begin{thm}
Suppose $S_{r}$ is the coordinate sphere of radius $r$ in an end of an asymptotically flat three-manifold $(M, g_{ij}, p_{ij})$ and $(E, P_1, P_2, P_3)$ is the ADM energy-momentum four vector of this end, then

\[\lim_{r\rightarrow \infty} E(S_{r}, X_{r}, T_0)=\sqrt{1+|a|^2} E+\sum_{i=1}^3 a^i P_i\] where $X_r$ is the (unique) isometric embedding of $S_r$ into $\R^3\subset \R^{3,1} $ and $T_0=(\sqrt{1+|a|^2}, a^1, a^2, a^3)$ is an arbitrary constant timelike unit vector.

\end{thm}

\begin{proof} Denote by $e_0$ the future timelike unit normal of the hypersurface $M$ and $\nu$ the unit outward normal of the coordinate sphere $S_r$.
Let $(y^1, y^2, y^3)$ be the asymptotically flat coordinates on the end. $S_r$ is given by $(y^1)^2+(y^2)^2+(y^3)^2=r^2$ and we denote the embedding of $S_r$ into $M$ by $Y$.  Since $p_{ij}=O(\frac{1}{r^2})$, we have $\langle H, e_0\rangle=O(\frac{1}{r^2})$. It is known that $\langle H, \nu\rangle =\frac{2}{r}+O(\frac{1}{r^2})$ (see for example \cite{fst}).
Since $ H=\langle H, \nu \rangle \nu-\langle H, e_0\rangle e_0$, we estimate
\[|H|-|\langle H, \nu\rangle|=O(\frac{1}{r_3}).\] Therefore,
\[\lim_{r\rightarrow \infty} \int_{S_r} |H_0|-|H| dv_r=\lim_{r\rightarrow \infty} \int_{S_r} |H_0|-|\langle H, \nu\rangle| dv_r,\] i.e, the Brown-York energy and the Liu-Yau energy have the same limit at spatial infinity. It is known that (see for example \cite{fst} and the reference therein) the Brown-York energy approaches the ADM energy $E$ at spatial infinity.

Now it suffices to prove
\[\sum_{i=1}^3 a^i P_i=\sum_{i=1}^3 a^ip_i=\frac{1}{8\pi}\int_{\Sigma_{r}} \langle\nabla^{N}_{\nabla\tau} \frac{{J}}{|{H}|}, \frac{{H}}{|{H}|}\rangle  dv_{{r}}.\] By definition, the ADM momentum is
\[\sum_{i=1}^3 a^iP_i=\frac{1}{8\pi} \int_{S_r} p(a^i\frac{\partial}{\partial y^i}, \nu)- ( tr p) \langle a^i\frac{\partial}{\partial y^i}, \nu\rangle dv_r.\] We decompose $a^i\frac{\partial}{\partial y^i}=(a^i\frac{\partial}{\partial y^i})^\top +\langle a^i\frac{\partial}{\partial y^i}, \nu\rangle \nu$ and the integrand becomes
\[p((a^i\frac{\partial}{\partial y^i})^\top, \nu)+\langle a^i\frac{\partial}{\partial y^i}, \nu\rangle (p(\nu, \nu)- ( tr p)).\]
By the definition of the mean curvature vector $H$, we obtain $p(\nu, \nu)- ( tr p)=\langle H, e_0\rangle$.
Therefore the ADM momentum term is
\begin{equation}\label{a^iP_i}\sum_{i=1}^3 a^iP_i=\frac{1}{8\pi} \int_{S_r}  \langle \nabla^N_{(a^i\frac{\partial}{\partial y^i})^\top} e_0, \nu\rangle+ \langle H, e_0 \rangle\langle a^i\frac{\partial}{\partial y^i}, \nu\rangle dv_r.\end{equation}

Now we turn to the limit of the quasilocal energy momentum. We can express the normal vector fields $H$ and $J$ in terms of $\nu$ and $e_0$ as
\[H=\langle H, \nu \rangle \nu-\langle H, e_0\rangle e_0 \text{ and } J=-\langle H, \nu\rangle e_0+\langle H, e_0\rangle \nu.\] We compute
\[\langle\nabla^{N}_{\nabla\tau} \frac{{J}}{|{H}|}, \frac{{H}}{|{H}|}\rangle=-\nabla \tau \cdot \nabla \sinh^{-1}(\frac{\langle H, e_0\rangle}{|H|})-\langle \nabla^N_{\nabla \tau} e_0, \nu\rangle. \] Integrating by parts gives
\[\frac{1}{8\pi}\int_{S_{r}} \langle\nabla^{N}_{\nabla\tau} \frac{{J}}{|{H}|}, \frac{{H}}{|{H}|}\rangle  dv_{{r}}=
\frac{1}{8\pi}\int_{S_{r}} -\langle \nabla^N_{\nabla \tau} e_0, \nu\rangle +\Delta \tau \sinh^{-1}(\frac{\langle H, e_0\rangle}{|H|}) dv_r.\]

Plug in (\ref{Delta_tau}), the second integrand on the right hand side becomes
\[ \langle e^{H_0}_3, T_0\rangle |H_0| \sinh^{-1}(\frac{\langle H, e_0\rangle}{|H|}).\]

Recall the asymptotics \[\langle H, e_0\rangle =O(\frac{1}{r^2}), |H|=\frac{2}{r}+O(\frac{1}{r^2}), |H_0|=\frac{2}{r}+O(\frac{1}{r^2})\] and $\sinh^{-1} x\sim x $ if $x<<1$. We see the limit of $\frac{1}{8\pi}\int_{S_{r}} \langle\nabla^{N}_{\nabla\tau} \frac{{J}}{|{H}|}, \frac{{H}}{|{H}|}\rangle  dv_{{r}}$ is the same as 
\begin{equation}\label{lim}\begin{split}\lim_{r\rightarrow \infty}\frac{1}{8\pi}\int_{S_{r}} -\langle \nabla^N_{\nabla \tau} e_0, \nu\rangle + \langle H, e_0\rangle \langle e^{H_0}_3, T_0\rangle  dv_r.\end{split}\end{equation}

Now we can compare (\ref{a^iP_i}) and (\ref{lim}).
Write out the tangential part of $a^i\frac{\partial}{\partial y^i}$,
\[(a^i\frac{\partial}{\partial y^i})^\top=\langle a^i\frac{\partial }{\partial y^i}, \frac{\partial Y}{\partial u^a}\rangle \sigma^{ab} \frac{\partial Y}{\partial u^b}=a^i g_{ij} \frac{\partial Y^j}{\partial u^a}\sigma^{ab} \frac{\partial Y}{\partial u^b}.\] On the other hand, as $\tau=-a^i X^i$, and the push-forward of $\nabla \tau$ becomes \[\nabla \tau=-a^i \frac{\partial X^i}{\partial u^a}\sigma^{ab} \frac{\partial Y}{\partial u^b}.\] The isometric embeddings satisfy (see for example \cite{fst})
\[|X^i-Y^i|=O(1) \text{ and } |e^{H_0}_3-\nu|=O(\frac{1}{r}).\]
From these, we deduce that (\ref{lim}) is the same as the limit of the right hand side of (\ref{a^iP_i}) and the theorem is proved.

\end{proof}

\section{Explicit computation in a boosted slice of Schwarzchild's solution}

In this section, we compute the limit of quasilocal energy-momentum for coordinate spheres of a boosted slice of Schwarzchild's solution.
\subsection{Asymptotics of the geometry of coordinate spheres} Let $(y^0, y^1, y^2, y^3)$ be the isotropic coordinates of Schwarzchild's solution in which the spacetime metric is of the form:
\[G_{\alpha\beta}d y^\alpha dy^\beta=-\frac{1}{F^2}(dy^0)^2+\frac{1}{G^2} \sum_{i=1}^3 (dy^i)^2\] with
\[F^2=\frac{(1+\frac{M}{2\rho})^2}{(1-\frac{M}{2\rho})^2},\,\,\, G^2=\frac{1}{(1+\frac{M}{2\rho})^4}\] and $\rho^2=\sum_{i=1}^3 (y^i)^2$. Given $\gamma>0$ and $\beta$ which satisfy $\gamma^2-\beta^2\gamma^2=1$, consider the surface $\Sigma_{r_0}$ defined by
\[\gamma y^0-\beta\gamma y^3=0\] and
\[(y^1)^2+(y^2)^2+(\gamma y^3-\beta\gamma y^0)^2=r^2_0\] with $r_0\rightarrow \infty$. With the coordinate change $(y^0)'=\gamma y^0-\beta\gamma y^3$
and $(y^3)'=\gamma y^3-\beta\gamma y^0$, these surfaces are coordinate spheres of radius $r_0$ in the asymptotically flat slice $\gamma y^0-\beta\gamma y^3=0$

We parametrize the 2-surfaces $\Sigma_{r_0}$ by
\[\begin{split} y^0& =\beta\gamma r_0\cos\theta\\
y^1&=r_0\sin\theta\sin\phi\\
y^2&=r_0\sin\theta \cos\phi\\
y^3&=\gamma r_0 \cos\theta.\end{split}\]
Denote the embedding of $\Sigma_{r_0}$ into Schwarzchild's solution by $Y=(y^0, y^1, y^2, y^3)$. In terms of local coordinates $u^1=\theta$ and $u^2=\phi$ on the surface, the induced metric on $\Sigma_{r_0}$ is
 \begin{equation}\label{ind_metric} \sigma_{ab}=r_0^2 [1+\frac{2M}{\rho}(1+2\beta^2\gamma^2\sin^2\theta)]d\theta^2+ r_0^2(1+\frac{2M}{\rho})\sin^2\theta d\phi^2 +O(r_0), \end{equation} and
 \begin{equation}\label{ind_volume_form}\sqrt{\det \sigma_{ab}}=r_0^2|\sin\theta|[1+\frac{2M}{\rho}(1+\beta^2\gamma^2\sin^2\theta)]+O(r_0).\end{equation}

The mean curvature vector ${H}=H^\gamma\frac{\partial}{\partial y^\gamma}$ of $\Sigma_{r_0}$ is by definition:
\[H^\gamma=\sigma^{ab}(\frac{\partial^2 y^\gamma}{\partial u^a\partial u^b}+ \Gamma_{\alpha\beta}^\gamma  \frac{\partial y^\alpha}{\partial u^a}  \frac{\partial y^\beta}{\partial u^b})(\delta_\beta^\gamma-\Pi_{\beta}^{\gamma})\] where  $\Gamma_{\alpha\beta}^\gamma$ are the Christoffel symbols of the metric $G_{\alpha\beta}$ and $\Pi_{\beta}^{\gamma}=G_{\beta\alpha} \sigma^{ab} \frac{\partial y^\alpha}{\partial u^a} \frac{\partial y^\gamma}{\partial u^b}$
 is the projection operator on the tangent space of $\Sigma_{r_0}$. The asymptotic expansion of $\Gamma_{\alpha\beta}^\gamma$ can be computed from the asymptotic expansion of $G_{\alpha\beta}$.

Denote by $\tilde{y}^\alpha=\frac{y^\alpha}{r_0}$ and $\tilde{\rho}=\frac{\rho}{r_0}$, which are both scaling invariant now. We shall use the following frames along $\Sigma_{r_0}$ to express the mean curvature vector:
\[\begin{split} N&=\tilde{y}^\alpha \frac{\partial}{\partial y^\alpha},\\
B&=\gamma(\frac{\partial}{\partial y^0}+\beta \frac{\partial}{\partial y^3}), \text{ and}\\
T&=\frac{\partial \tilde{y}^\alpha}{\partial \theta}\frac{\partial}{\partial y^\alpha}.\end{split}\]

We notice that $T$ is a tangent vector field to $\Sigma_{r_0}$ while $N$ and $B$ are only asymptotically normal in the sense that
$\langle N, T\rangle=O(\frac{1}{r_0})$ and $\langle B, T\rangle=O(\frac{1}{r_0})$. We also have $\langle N, N\rangle=1+O(\frac{1}{r_0})$, $\langle B, B\rangle=-1+O(\frac{1}{r_0})$, and $\langle T, T\rangle=1+O(\frac{1}{r_0})$.

A straightforward calculation gives
\begin{lem}
\[\begin{split}{H}&=\frac{-2}{r_0}N+\frac{1}{r_0^2}(\mathfrak{n} N+\mathfrak{t} T+\mathfrak{b} B)+O(\frac{1}{r_0^3})\end{split}\] with
\[\begin{split} \mathfrak{n}&=\frac{M}{\tilde{\rho}^3}(6+6\beta^2\gamma^2+2\beta^4\gamma^4\sin^2\theta\cos^2\theta),\\
\mathfrak{t}&=\frac{M}{\tilde{\rho}^3}(-8\tilde{\rho}^2)(\beta^2\gamma^2\sin\theta\cos\theta), \text{ and}\\
\mathfrak{b}&=\frac{M}{\tilde{\rho}^3}(2\beta\gamma^2\cos\theta)(\beta^2\gamma^2\sin^2\theta-1).\end{split}\]
\end{lem}

From here, we compute the norm of $|{H}|$,
\begin{pro}\label{mean_curv_norm}
\[|{H}|=\frac{2}{r_0}+\frac{1}{r_0^2}[\frac{2M}{\tilde{\rho}}(1+2\beta^2\gamma^2\cos^2\theta)-\mathfrak{n}]+O(\frac{1}{r_0^3}).\]
\end{pro}
Let ${J}$ be the future-directed timelike normal vector that is dual to ${H}$ along the light cone in the normal bundle.

\begin{lem}
${J}$ is given by
\[\begin{split}&\frac{2}{r_0}B-\frac{1}{r_0^2}[\frac{-2M(1+2\beta^2\gamma^2\cos^2\theta+\gamma^2+\beta^2\gamma^2)}{\tilde{\rho}}+\mathfrak{n}]B\\
&+\frac{1}{r_0^2}[\frac{8M}{\tilde{\rho}}(\beta\gamma^2\sin\theta)T-(\mathfrak{b}+\frac{8M\beta\gamma^2\cos\theta}{\tilde{\rho}})N].
\end{split}\]
\end{lem}
\begin{proof}
The coefficients of ${J}$ are determined by the following equations:

\[-\langle {J}, {J}\rangle=\langle {H}, {H}\rangle,\] \[\langle {J}, {H}\rangle=0\] and \[\langle {J}, T\rangle=0.\]

\end{proof}

With this explicit formula, we  compute the coefficients of the connection form of the normal bundle in the mean curvature vector gauge:

\begin{pro}\label{order_mean_curv_gauge}
\[\langle \nabla_{\frac{\partial Y}{\partial \theta}} {J}, {H}\rangle=\frac{1}{r_0^3}(2\mathfrak{b}'+\frac{8M}{\tilde{\rho}}\beta\gamma^2\sin\theta)+O(\frac{1}{r_0^4})\] and
\[\langle \nabla_{\frac{\partial Y}{\partial \phi}} {J}, {H}\rangle=O(\frac{1}{r_0^4}).\]
\end{pro}
It turns out the second term does not contribute to the limit of the quasilocal energy.
\subsection{Total mean curvature of isometric embedding}

We consider the isometric embedding of a general axially symmetric metric into $\R^3$.  The metric  is of the form
\[r_0^2 P^2(r_0, \theta) d\theta^2+r_0^2Q^2(r_0, \theta) \sin^2\theta d\phi^2\] with
\[P(r_0, \theta)=1+O(\frac{1}{r_0}),\text{ and } Q(r_0, \theta)=1+O(\frac{1}{r_0}).\]

Suppose the isometric embedding is given by \[X=(u(r_0, \theta) \sin\phi, u(r_0, \theta)\cos\phi, v(r_0, \theta)).\] Thus
\[\frac{\partial X}{\partial \theta}=(\frac{\partial u}{\partial \theta} \sin\phi, \frac{\partial u}{\partial \theta}\cos\phi, \frac{\partial v}{\partial  \theta})\] and \[\frac{\partial X}{\partial \phi}=(u \cos\phi, -u \sin \phi, 0).\]

It is not hard to see

\begin{lem}
$u$ and $v$ are given by
\[(\frac{\partial u}{\partial \theta})^2+(\frac{\partial v}{\partial \theta})^2=r_0^2P^2\] and
\[u^2=r_0^2Q^2\sin^2\theta.\]
\end{lem}

\begin{pro}The mean curvature of the isometric embedding of the metric $r_0^2 P^2(r_0, \theta) d\theta^2+r_0^2Q^2(r_0, \theta) \sin^2\theta d\phi^2$ into $\R^3$ is given by
\[\begin{split}H_0&=-(\frac{1}{r_0^3P^3})(\frac{\partial^2 u}{\partial \theta^2} \frac{\partial v}{\partial \theta}+\frac{\partial^2 v}{\partial \theta^2}\frac{\partial u}{\partial \theta})+(\frac{1}{r_0^2PQ\sin\theta}) \frac{\partial v}{\partial \theta}.\end{split}\] where
\[ u(r_0, \theta)=r_0Q(r_0, \theta)\sin \theta,\] and
\[(\frac{\partial v}{\partial \theta})^2=r_0^2P^2-(\frac{\partial u}{\partial \theta})^2=r_0^2[P^2-(\frac{\partial Q}{\partial \theta} \sin \theta+Q \cos\theta)^2].\]
\end{pro}

Now suppose
\[P=1+\frac{p}{r_0}+O(\frac{1}{r_0^2}), p=p(\theta)\] and
\[Q=1+\frac{q}{r_0} +O(\frac{1}{r_0^2}), q=q(\theta).\] The asymptotic expansion of the mean curvature is found to be
\begin{equation}\label{total_mean_axil}\begin{split}H_0&=\frac{2}{r_0}-\frac{1}{r^2}(2p+\frac{\cos\theta}{\sin\theta}(2q'-p')+q'')\end{split}.\end{equation}

Comparing with (\ref{ind_metric}), we deduce that in our case
\[p=\frac{M}{\tilde{\rho}}(1+2\beta^2\gamma^2\sin^2\theta) \text{ and } q=\frac{M}{\tilde{\rho}}.\] $u$ and $v$ can be solved explicitly:
\begin{equation}\label{u_and v}u=r_0\sin\theta+\frac{M}{\tilde{\rho}} \sin\theta \text{ and } v=r_0\cos\theta +\frac{M}{\tilde{\rho}}\cos\theta +2M\beta\gamma \sinh^{-1}(\beta \gamma \cos\theta).\end{equation}

Plug in the expression of $p$ and $q$ into (\ref{total_mean_axil}) and integrate by parts, we obtain

\begin{pro}\label{total_mean_ref}
\[\int_{\Sigma_{r_0}} H_0 dv_{r_0}=8\pi r_0+2\pi M\int_0^{2\pi} \frac{1+\beta^2\gamma^2\sin^2\theta}{\tilde{\rho}} |\sin\theta| d\theta +O(\frac{1}{r_0}).\]
\end{pro}

This calculation is compatible with Lemma 2.4 in \cite{fst}.
\subsection{Evaluating the quasilocal energy}

We are ready to compute the limit of the Liu-Yau mass:
\begin{pro}
\[\int_{\Sigma_{r_0}}( H_0 -|{H}|) dv_{r_0}=8\pi \gamma M +O(\frac{1}{r_0}).\]
\end{pro}

\begin{proof} Combine Proposition \ref{mean_curv_norm} and Proposition \ref{total_mean_ref}, we obtain
\[\int_{\Sigma_{r_0}}( H_0 -|{H}|) dv_{r_0}=\pi M\int_0^{2\pi} \frac{2+4\beta^2\gamma^2-6\beta^2\gamma^2\cos^2\theta-4\beta^4\gamma^4\cos^4\theta}{\tilde{\rho}} |\sin\theta| d\theta+O(\frac{1}{r_0}).\]

The integral can be evaluate by the substitution $\beta\gamma \cos\theta=\sinh y$.
\end{proof}
Now we turn to the momentum part. Suppose $T_0=(\sqrt{1+|a|^2}, a^1, a^2, a^3)$, $|a|^2=\sum_{i=1}^3 (a^i)^2$ is a future timelike unit vector and the isometric embedding into $\R^3\subset \R^{3,1}$ is given by $X=(0, u\sin\phi, u\cos\phi, v)$. We know from (\ref{u_and v}) that $u=r\sin\theta+O(1)$ and $v=r\cos\theta +O(1)$. The gradient of $\tau$ is given by $\nabla \tau=\frac{\partial \tau}{\partial u^a} \sigma^{ab} \frac{\partial Y}{\partial u^b}$.  We compute
\begin{equation}\label{momentum}\begin{split}\int_{\Sigma_{r_0}} \langle \nabla_{\nabla \tau} \frac{J}{|H|}, \frac{H}{|H|}\rangle dv_{r_0}
&=-a^1\int_{\Sigma_{r_0}}\frac{1}{|H|^2} (u\sin\theta)' \sigma^{\theta\theta} \langle \nabla_{\frac{\partial Y}{\partial \theta}} {J}, {H}\rangle dv_{r_0}\\
&-a^2\int_{\Sigma_{r_0}}\frac{1}{|H|^2} (u\cos\theta)' \sigma^{\theta\theta} \langle \nabla_{\frac{\partial Y}{\partial \theta}} {J}, {H}\rangle dv_{r_0}\\
&-a^3\int_{\Sigma_{r_0}}\frac{1}{|H|^2} v' \sigma^{\theta\theta} \langle \nabla_{\frac{\partial Y}{\partial \theta}} {J}, {H}\rangle dv_{r_0}+O(\frac{1}{r_0}).\end{split}\end{equation}

These integrals can be evaluated and we obtain

\begin{pro}
\[\int_{\Sigma_{r_0}} \langle \nabla^N_{\nabla \tau} \frac{J}{|H|}, \frac{H}{|H|}\rangle dv_{r_0}=a^3 8\pi \beta\gamma M+O(\frac{1}{r_0}).\]
\end{pro}
\begin{proof}
 By Proposition \ref{order_mean_curv_gauge}, $\langle \nabla_{\frac{\partial Y}{\partial \theta}} {J}, {H}\rangle$ is of the order $\frac{1}{r_0^3}$ while $u$ and $v$ are both of order $r_0$, we have \[\begin{split}&\int_{\Sigma_{r_0}} \frac{1}{|H|^2} (u\sin\theta)' \sigma^{\theta\theta} \langle \nabla_{\frac{\partial Y}{\partial \theta}} {J}, {H}\rangle dv_{r_0}\\
&=\int_0^{\pi}\int_0^{2\pi}\frac{r_0^2}{4} (r_0 \sin^2\theta)' \frac{1}{r_0^2} \langle \nabla_{\frac{\partial Y}{\partial \theta}} {J}, {H}\rangle r_0^2|\sin\theta|d\theta d\phi\\
&=\frac{\pi}{4}\int_0^{2\pi} ( \sin^2\theta)' [{r_0^3} \langle \nabla_{\frac{\partial Y}{\partial \theta}} {J}, {H}\rangle] |\sin\theta|d\theta . \end{split}\]
Therefore the first integral on the right hand side of (\ref{momentum}) is
\[-a^1\frac{\pi}{4}\int_0^{2\pi} ( \sin^2\theta)' (2\mathfrak{b}'+\frac{8M}{\tilde{\rho}}\beta\gamma^2\sin\theta) |\sin\theta|d\theta \] where
\[\mathfrak{b}=\frac{M}{\tilde{\rho}^3}(2\beta\gamma^2\cos\theta)(\beta^2\gamma^2\sin^2\theta-1), \] and the second one is
\[-a^2\frac{\pi}{4}\int_0^{2\pi} ( \sin\theta\cos\theta)' (2\mathfrak{b}'+\frac{8M}{\tilde{\rho}}\beta\gamma^2\sin\theta) |\sin\theta|d\theta.\]

Both integrate to zero as they are of the form $\int_0^{2\pi} (\cos\theta) F(\cos^2\theta) |\sin\theta|d\theta$ or $\int_0^{2\pi} (\sin\theta) F(\cos^2\theta)|\sin\theta|d\theta$ for some smooth function $F$ of $\cos^2\theta$. The last integral becomes
\[ a^3 \frac{ \pi}{4}\int_0^{2\pi} \sin\theta (2\mathfrak{b}'+\frac{8M}{\tilde{\rho}}\beta\gamma^2\sin\theta) |\sin\theta|d\theta \] which can be simplified as \[a^3 4\pi M\beta\gamma^2 \int_0^{\pi} \frac{\sin\theta}{\tilde{\rho}^3}d\theta.\]
Using the same substitution $\beta\gamma\cos\theta=\sinh y$,  the integral is
\[a^3 8\pi \beta \gamma M.\]
\end{proof}

Therefore the limit of the quasilocal energy (\ref{qle}) is
\[(\sqrt{1+|a|^2})\gamma M+a^3 \beta\gamma M.\] Recall that $\gamma^2-\beta^2\gamma^2=1$. Minimizing this expression among all $T_0=(\sqrt{1+|a|^2}, a^1, a^2, a^3)$, we see the minimum is achieved at $(\gamma, 0, 0, -\beta\gamma)$ and the minimum value is $M$. The limit of the quasilocal energy-momentum is thus $M(\gamma, 0, 0, -\beta\gamma)$.

\end{document}